\documentclass[a4paper, 12pt,oneside,reqno]{amsart}
\usepackage[a4paper]{geometry}
\usepackage[english]{babel}
\usepackage[T1]{fontenc}
\usepackage{amsfonts}
\usepackage{mathrsfs}
\usepackage{tikz}
\usetikzlibrary{arrows,shapes,snakes,automata,backgrounds,petri,through,positioning}
\usetikzlibrary{intersections}

\usepackage[matrix,arrow]{xy}
\usepackage{amssymb,amscd,amsthm,amsmath}
\usepackage[a4paper]{geometry}
\usepackage{amsmath}
\usepackage{amssymb}
\usepackage{amsthm}
\usepackage[colorlinks=true, allcolors=blue]{hyperref}

\newtheorem{theorem}{Theorem}[section]
\newtheorem{lemma}[theorem]{Lemma}

\newtheorem{proposition}[theorem]{Proposition}
\newtheorem{corollary}[theorem]{Corollary}
\theoremstyle{definition}
\newtheorem{definition}[theorem]{Definition}
\newtheorem{example}[theorem]{Example}

\renewcommand{\phi}{\varphi}

\title{On logarithmic Poisson cohomology of a degenerate Poisson bivector in affine plane.}
\author{Kamtila  Kari}
\address{ University of Maroua}
\email{kamtilakari@gmail.com}
\author{Iskamle Bruno}
\address{ University of Maroua}
\email{brunoiskamle@gmail.com}
\author{Diekouam Fotso L. E.}
\address{ University of Maroua}
\email{diekouamluc@gmail.com}
\author{Tcheka Calvin}
\address{ University of Dschang}
\email{jtcheka@yahoo.fr}
\thanks{Sincerely, the authors dedicate this paper to the memory of \textbf{Professor Joseph Dongho}, who passed aways on \textbf{26 February 2026} as this work was nearing completion. His intellectual rigour, generosity and lasting influence continue to inspire us.}   
\begin{document}

\begin{abstract}
In this paper, we show that for a given degenerate bivector $\pi=y^n\partial_{x}\wedge \partial_{y}$ with $n>1$, the classical Poisson cohomology group and the logarithmic Poisson cohomology group along the ideal $\mathcal{I}=y^n\mathbb{F}[x,y]$ are isomorphics in every degree.  
This result follows from determination of the logarithmic Hamiltonian opreator and the logarithmic Poisson cochain complex  in order to compute the cohomological invariants associated to  $\pi$. $\mathbb{F}$ is the field of characteristic zero.
\end{abstract}
\maketitle
\section{Introduction}

A Poisson bivector on an $m$-dimensional manifold $M$ is any
$\pi\in\mathfrak{X}^2(M)=\Gamma(\wedge^2 TM)$  such that;
$[\pi,\pi]_{SN}=0$, where $[.,.]_{SN}$ is the Schouten-Nijenhuis bracket of
multivector fields. In algebraic point of view, since each Poisson
structure on $M$ induce a bi-derivation Lie bracket $\{f,
g\}=\pi(df\otimes dg)$ for all $f,g\in \mathcal{C}^\infty(M),$ a
Poisson algebra is any commutative and associative algebra $A$ over an
unitary ring $R$ endowed with a bi-derivation Lie bracket $\{.,.\}$. One of the first Poisson bracket on the algebra of smooth
functions on $\mathbb{R}^{2n}$ is defined by

\begin{equation}\label{E00}
 \{f,g\}=\underset{i=1}{\overset{n}\sum}(\dfrac{\partial f}{\partial p_i}\dfrac{\partial g}{\partial q^i}-\dfrac{\partial f}{\partial q^i}\dfrac{\partial g}{\partial p_i})
\end{equation}
     This Poisson structure play a fundamental role in the analytical mechanics.
  And the 
   Poisson cohomology together with Poisson manifolds was introduced by A. Lichnerowicz (see \cite{AL}). It also plays a fundamental role in the deformation of Poisson structures, in the deformation quantization(see \cite{FB}) and in geometric quantization(see \cite{DJ1}). \\
   
   Some studied cases of classical Poisson cohomology that we can list include for example, those of C. Roger and P. Vanhaecke (in \cite{CRPV}), A. Pichereau (see \cite{PA}), P. Monnier (see \cite{PM}), 	N. Nakanishi (see \cite{NN}) and B. Iskamle, J. Dongho, B. Ndombol(see \cite{BI}). We cannot mention them without citing J. Dongho (see \cite{DJ1}), who is the initiator of logarithmic Poisson cohomology along certain divisors so-called principals, among which divisor $\{y^n=0, n>1\}$ does not appear. This is why we think it's necessary to investigate this case.\\
   
Doing so is substantially important beacuse there are very few explicit examples of logarithmic Poisson cohomology calculations, so it may be interest readers.
The motivation for this case is to compare the cohomological invariants with Bruno Iskamle's case(see \cite{BI})   on the classical Poisson cohomology of the Poisson bracket $\{x,y\}_{0}=y^n, n\geq 2$. The case where Poisson cochain complex was predefined by Claude Roger and Pol Vanhaecke(see \cite{CRPV}), it is not for the logarithmic Poisson cochain complex. \\

Our main results on the logarithmic Poisson cohomology assiciated to the Poisson bivector $\pi = y^n\partial x \wedge \partial y$ where $n>1$ are:
\begin{enumerate}
\item[$\bullet$] The degree 0 is $H_{log}^{0}(\tilde{\mathcal{P}}) \simeq \mathbb{F}$,
\item[$\bullet$] The degree 1 is $H_{log}^{1}(\tilde{\mathcal{P}}) \simeq  \mu (\overset{n-2}{\underset{i=0}{\bigoplus}}y^{i} \mathbb{F}[x])\oplus ( \mathbb{F}_{n-1}[y]\times 0)$, where $\mu$ is the  $\mathbb{F}$-linear map given by  $\mu : \mathcal{A} \longrightarrow \mathcal{A}\times \mathcal{A}; b\mapsto \left(\int ((n-1)b - y\partial_{y}b)dx; b\right)$,
\item[$\bullet$] The degree 2 is $H_{log}^{2}(\tilde{\mathcal{P}}) \simeq \overset{n-2}{\underset{i=0}{\bigoplus}}y^{i} \mathbb{F}[x]$,
\item[$\bullet$] The degree $k>2$ is $H_{log}^{k}(\tilde{\mathcal{P}}) \simeq 0$, for all $k>2$. 
\end{enumerate}
We deduce according to \cite{BI}'s results and those of this work that, for the above degenerate bivector $\pi$, it follows that there exists an isomorphism $\Pi^{*}: H_{Pois}^{*}(\mathcal{P}) \longrightarrow  H_{log}^{*}(\tilde{\mathcal{P}})$ where $H_{Pois}^{k}(\mathcal{P})$ denote  the classical Poisson cohomology of degree $k$. It is not always the case for any Poisson bivector $\pi = \{x,y\}\partial _{x}\wedge \partial _{y}$.\\

 The paper is stuctured as follows:
 After the introduction which establishes the context, the second step  concerns the preliminaries on some properties of the Poisson structure, the Koszul bracket and the cochain complex,  
the third movement  is devoted to determination of the logarithmic cochain complex,
 and we completes in fourth part with  the computation of the logarithmic Poisson
  cohomology of  $ (\mathbb{F}[x,y],\{x,y\}=y^n, \tilde{H})$, $n>1$ along the ideal $\mathcal{I}=y^n\mathbb{F}[x,y]$ where $\tilde{H}$ is the associated logarithmic Hamiltonian operator.  
 
\section{Preliminaries}\label{Rapp}
\begin{definition} \cite{CL}
A Poisson algebra is an $\mathbb{F}$-vector space $\mathcal{A}$ equipped with two operations $(f,g)\mapsto f.g$ and $(f,g)\mapsto \{f,g\}$ such that
  \begin{enumerate}
  \item[i.] $(\mathcal{A},.)$ is a commutative associative algebra over $\mathbb{F}$ with unit 1;
  \item[ii.] $(\mathcal{A},\{.,.\})$ is a Lie algebra of $\mathbb{F}$;
    \item[iii.] The two operations are compatibles in the sence that for all $a,b,c\in \mathcal{A}$,
    \begin{equation}\label{R2}
\{a.b, c\}=a.\{b,c\}+b.\{a,c\}.
    \end{equation}
 \end{enumerate}
\end{definition}
\begin{example} Let  $\mathcal{A} = \mathbb{F}[x,y]$  and
$\{f, g\}_{0} =  \{x,y\}_{0}  \left(\dfrac{\partial f}{\partial x } \dfrac{\partial g}{\partial y } - \dfrac{\partial f}{\partial y } \dfrac{\partial g}{\partial x }\right)$ for all $f,g\in \mathcal{A}$. So $(\mathcal{A}, \{.,.\}_{0})$ is the Poisson algebra. 
\end{example}
 According to the Leibniz rule given in $(\ref{R2})$  we deduce that for all $a\in \mathcal{A}$, the map  $ad_{a} :\mathcal{A}\rightarrow \mathcal{A}$, $b\longmapsto \{a,b\}$ is a derivation on  $\mathcal{A}$. Moreover for all  $a,b\in \mathcal{A}$ we have:
 \begin{eqnarray*}
 ad_{a.b}(x) = a.\{b,x\}+ b.\{a,x\}=a. ad_{b}(x)+b. ad_{a}(x),
 \end{eqnarray*}
   The map $ad :\mathcal{A}\rightarrow Der_{\mathcal{A}},a\longmapsto \{a,.\} $ is a derivation of $\mathcal{A}$  with values in the  $\mathcal{A}$-module $Der_{\mathcal{A}}$. According to the universal property of $(\Omega_{\mathcal{A}}, d)$, the derivation $ad:\mathcal{A} \rightarrow Der_{\mathcal{A}}$ induce an homomorphism of $\mathcal{A}-$modules called Hamiltonian map defined as follows
 \begin{eqnarray}
 H: \Omega_{\mathcal{A}} \rightarrow Der_{\mathcal{A}}, dx\longmapsto H(dx)=\{x,.\}
 \end{eqnarray} such that the following diagram commute
 \begin{eqnarray}
 \xymatrix{\mathcal{A}\ar[r]^d\ar[dr]_{ad}&\Omega_{\mathcal{A}}\ar[d]^H\\
                                   &             Der_{\mathcal{A}}}
 \end{eqnarray}

\begin{example} For $\{x, y\} = y^n$ and $ \mathcal{A}= \mathbb{F}[x,y]$, the Hamiltonian operator is $H: \Omega _{\mathcal{A}} \longrightarrow Der _{\mathcal{A}}$ such that  $df\mapsto H_{f} = y^n (\partial_{x}f\partial_{y}  - \partial_{y}f \partial_{x})\in Der _{\mathcal{A}}$. 
\end{example}
Let  $\{h=0\}$ the divisor on $\mathbb{F}^{p}$ with coordinate system $x_{1},...,x_{p}$ and ideal $\mathcal{I} = h\mathcal{A}$.
 \begin{definition} \cite{DJ1}
A derivation $D$ of $\mathcal{A}$ is said to be  logarithmic along $\mathcal{I}$ if we have $D(\mathcal{I}) \subset \mathcal{I}$. Denote by $Der _{\mathcal{A}}(log \mathcal{I})$ the set of all  derivation logarithmics along $\mathcal{I}$ and  we denote by $\Omega^{1} _{\mathcal{A}}(log \mathcal{I})$ it dual, which is the  $\mathcal{A}$-module of  k\"{a}hler differential  1-form, logarithmic along $\mathcal{I}$.
\end{definition}

\begin{definition}
Let $h=0$ the hyperplan equation on affine plane $\mathbb{F}^{2}$ and $ \mathcal{A}= \mathbb{F}[x,y]$.
If the divisor $\{h=0\}$ is logarithmic along the ideal $\mathcal{I}$, the Poisson structure $\{x,y\} = h$ is said logarithmic along $\mathcal{I} =  h\mathcal{A}$.
\end{definition}

\begin{proposition} \cite{DJ1} $\label{P1}$
Each logarithmic Poisson structure along 
 $\mathcal{I}$ induce  an $\mathcal{A}$-module homomorphism $\tilde{H}$ from $\Omega ^{1}_{\mathcal{A}}(log \mathcal{I})$ to $Der _{\mathcal{A}}(log \mathcal{I})$ by $\tilde{H}(\dfrac{du}{u})=\dfrac{1}{u}H(du)$.
\end{proposition}

\begin{theorem}  \cite{DJ1} $\label{T2}$
Let $\{.,.\}$ be a logarithmic Poisson structure along 
 $\mathcal{I}$.
 The map $[.,.]$ given by $[.,.]: (\omega_{i},\omega_{j})\mapsto \mathcal{L}_{\tilde{H}(\omega_{i})}(\omega_{j}) + \mathcal{L}_{\tilde{H}(\omega_{j})}(\omega_{i}) - d \pi (\omega_{i},\omega_{j})$ is a Lie structure on $\Omega^{1} _{\mathcal{A}}(log \mathcal{I})$. $\pi$ is a $2$-form (or the Poisson bivector) associated to  $\{.,.\}$.
 
\end{theorem}
This bracket named \textbf{Koszul bracket}  imply the  following $corollary$ that we need some properties
to construct the logarithmic cochain complex associated to the bivector $\pi = y^n\partial_{x}\wedge \partial_{y}$.
\begin{corollary} \cite{DJ1} $\label{Cor1} $
Let $u,v\in \mathcal{I}^{*}$ and $[.,.]$ the \textbf{Koszul bracket} of logarithmic 1-forms.
\item[a.] $[\dfrac{du}{u}, \dfrac{dv}{v}] = d(\dfrac{1}{uv}\{u,v\})$,
\item[b.] $[\dfrac{du}{u}, dv] = d(\dfrac{1}{u}\{u,v\})$,
\item[c.] $[du, \dfrac{dv}{v}] =d(\dfrac{1}{v}\{u,v\})$,
\item[d.] $[du, dv] =  d\{u,v\}$.
\end{corollary}

 \begin{definition} \cite{AJ}
 An algebraic cochain complex consists of a sequence of
 $\mathcal{A}$-modules $\textbf{ C}^{*}$ and $\mathcal{A}$-module
 homomorphisms $d^{q}: C^{q-1} \longrightarrow C^{q}$ such that
 $d^{q}\circ d^{q-1}=0$, $q\in \mathbb{Z}$. It is denoted by
 $(\textbf{ C}^{*}, d^{*})$ and it induce the following sequence:
 \[\xymatrix{(\textbf{ C}^{*}, d^{*}):  ...\ar[r]^{d^{q-2}}& C^{q-2} \ar[r]^{d^{q-1}}& C^{q-1}  \ar[r]^{d^{q}}& C^{q}   \ar[r]^{d^{q+1}}& ...}\]
 \end{definition}
 $d:=d^{q}$ is the coboundary operator. Elements of $(C^{q}, d^{q})$
 are called $q$-cochain, those of $Z^{q}(\textbf{ C}^{*}, d^{*})=Ker
 d^{q}$ are called $q$-cocycles and those of $B^{q}(\textbf{ C}^{*},
 d^{*})= Imd^{q-1}$ are called $q$-coboundary. Since $d^{2}=0$, then
 $B^{q}(\textbf{ C}^{*},  d^{*})$ is submodule of  $Z^{q}(\textbf{
 C}^{*}, d^{*})$ and $H^{q-1}(\textbf{ C}^{*}, d^{*})=
 \dfrac{Z^{q}(\textbf{ C}^{*}, d^{*})}{B^{q}(\textbf{ C}^{*},
 d^{*})}$ is well defined and called the $(q-1)^{th}$ cohomology
 $\mathcal{A}$-module of $(\textbf{ C}^{*}, d^{*})$.
 
 \begin{example}
 Let $\mathcal{A} = \mathbb{F}[x,y]$.
 According to \cite{CRPV}, for each Poisson structure $\{x,y\}^{\varphi} = \varphi$ on affine plane $\mathbb{F}^2$ the cochain complex is given as follows:
  \[\xymatrix{0 \ar[r]^{d^{0}}& \mathcal{A} \ar[r]^{d^{1}}& \mathcal{A}^2  \ar[r]^{d^{2}}& \mathcal{A}  \ar[r]^{d^{3}}& 0}\] 
  where for all $f,g\in \mathcal{A}$, $d^1(f)= \varphi \left( \partial_{y}f, -\partial_{x}f\right)$,    $d^2(f,g)= \varphi \left( \partial_{x}f + \partial_{y}g \right) - f \partial_{x}\varphi - g\partial_{y}\varphi$ and $d^0 = d^3 = 0$.\\
  After computing for the case where the Poisson structure is given by $\{x,y\}^{\varphi}=y^n$, we have 
    $d^1(f)= y^n \left( \partial_{y}f, -\partial_{x}f\right)$,    $d^2(f,g)= y^n \left( \partial_{x}f + \partial_{y}g \right)  - ny^{n-1}g$. The  associated classical Poisson cohomology of degree $k$ is  $H_{Pois}^{k}(\mathcal{P}) = \dfrac{Kerd^{k+1}}{Imd^{k}}$(we can see it in \cite{BI}).
 \end{example}
 
 \begin{definition} \cite{AJ}
 A cochain morphism $f^{*}: (\textbf{ C}^{*}, d^{*}) \longrightarrow
 (\tilde{\textbf{ C}^{*}}, \tilde{d}^{*})$ consists of a sequence of
 the homomorphisms of $\mathcal{A}$-module $f^{q}: (C^{q}, d^{q})
 \longrightarrow (\tilde{C}^{q}, \tilde{d}^{q})$ that commute with
 the coboundary operators. That is equivalent to say that the squares
 of the following diagram:
  \[ \xymatrix{ ...\ar[r]^{d^{q-1}}& C^{q-1} \ar[d]^{f^{q-1}}\ar[r] ^{d^{q}} & C^{q}\ar[d]^{f^{q}}\ar[r]^{d^{q+1}}&C^{q+1} \ar[d]^{f^{q+1}}\ar[r] ^{d^3}& ...\\
       ...\ar[r]^{\tilde{d}^{q-1}} & \tilde{C}^{q-1} \ar[r] ^{\tilde{d}^{q}}& \tilde{C}^{q}  \ar[r]^{\tilde{d}^{q+1}} &  \tilde{C}^{q+1}  \ar[r]& ...}
      \]
 \end{definition} 

\section{Logarithmic Poisson's cochain complex}
For $\mathcal{A} = \mathbb{F}[x_{1},...,x_{p}]$.  Denote by 
      $\mathfrak{L}_{alt}( \Omega^{1}_{\mathcal{A}}(log\mathcal{I}),
       \mathcal{A}) = \underset{i\geq
       0}{\bigoplus}\mathfrak{L}_{alt}^{p}(\Omega^{1}_{\mathcal{A}}(log\mathcal{I}),
       \mathcal{A}) $  the $\mathcal{A}$-module of the $p$- multilinear
       alternate forms on $\Omega^{1}_{\mathcal{A}}(log\mathcal{I})$; $\mathcal{I}$  be an ideal of $\mathcal{A}$.
      
\begin{definition} $\label{5}$
  We name the logarithmic cochain complex of dimension $p$ associated to all alternate $p$-linear map $\tilde{H}: \Omega^{1}_{\mathcal{A}}(log \mathcal{I}) \longrightarrow Der_{\mathcal{A}}(log \mathcal{I})$, the following sequence:
  \begin{displaymath} \label{CC1}
       (\textbf{C}^{*},d^{*}):    ...\stackrel{d^{i}_{\tilde{H}}}{\longrightarrow}\wedge^{i}Der_{\mathcal{A}}(log \mathcal{I})
              \stackrel{d^{i+1}_{\tilde{H}}}{\longrightarrow} \wedge^{i+1}Der_{\mathcal{A}}(log \mathcal{I}) \stackrel{d^{i+2}_{\tilde{H}}}{\longrightarrow} ...
              \end{displaymath}
        where the differential 
                $d^{i}_{\tilde{H}}:  \mathfrak{L}^{i-1}_{alt}( \Omega^{1}_{\mathcal{A}}(log\mathcal{I}), \mathcal{A})
                \longrightarrow \mathfrak{L}_{alt}^{i}( \Omega^{1}_{\mathcal{A}}(log\mathcal{I}), \mathcal{A})$ is the logarithmic Poisson differential such that,
                for all  $f \in \mathfrak{L}_{alt}(\Omega^{1}_{\mathcal{A}}(log\mathcal{I}), \mathcal{A})$
                and $ \omega_{1},...,\omega_{p+1}\in \Omega^{1}_{\mathcal{A}}(log\mathcal{I})$:
                \begin{eqnarray*}
                  d^{i}_{\tilde{H}}f(\omega_{1},...,\omega_{p+1})
                   &=& \sum_{i=1}^{p+1}(-1)^{i-1}\tilde{H}(\omega_{i})f( \omega_{1},...,\hat{\omega}_{i},...,\omega_{p+1})+ \label{*} ~~~~~~~~~~~~~~~~~~  \\
                   & &\sum_{1\leq i\leq j \leq p+1} (-1)^{i+j}f([\omega_{i}, \omega_{j}]_{_{\Omega^{^{1}}_{\mathcal{A}}(log\mathcal{I})}},
                   \omega_{1},...,\hat{\omega}_{i},
                   ...,\hat{\omega}_{j},...,\omega_{p+1}).
                \end{eqnarray*}
                      
  \end{definition}
  The corresponding cohomology is the Lichnerowicz-Poisson cohomology  called the logarithmic Poisson cohomology of the logarithmic Poisson algebra $(\mathcal{A}, \{.,.\}, \tilde{H})$
          where $\tilde{H}$ is the  logarithmic Hamiltonian map. For all $a\in \Omega^{1}_{\mathcal{A}}(log \mathcal{I})$, it verifies the following compatibility:
            $$ [\omega_{i},a\omega_{j}]= \tilde{H}(\omega_{i})(a)\omega_{j} + a[\omega_{i},\omega_{j}].$$

For the remainder of this work, we will assume that $\mathcal{A} = \mathbb{F}[x,y]$, for all $n>1$ $\mathcal{I}=y^n\mathcal{A}$ the ideal of $\mathcal{A}$ and $\mathbb{F}_{n-1}[y]$ the vector space of  polynomial in variable $y$ with degree less than or equal to $n-1$.  
$\int b dx$ denotes the antiderivative of $b$ with respect to $x$.
The basis of the $\mathcal{A}$-modules $Der_{\mathcal{A}}(log \mathcal{I}) $  and $\Omega_{\mathcal{A}}^{1}(log \mathcal{I}) $  are respectively given by $\langle \delta^{1} = \partial_{x}, \delta^{2} = y\partial_{y} \rangle$ and $\langle\omega_{1}=dx, \omega_{2}=\dfrac{dy}{y} \rangle$.

\begin{proposition}
  The logarithmic  Hamiltonian operator $\tilde{H}$ induced by the  bivector $\pi = y^n \partial_{x} \wedge \partial_{y}$ 
 is
 $\tilde{H}: \Omega_{\mathcal{A}}^{1}(log\mathcal{I}) \longrightarrow
 Der_{\mathcal{A}}(log \mathcal{I})$  such that, for all $\omega =
 a\omega_{1} + b \omega_{2}\in \Omega_{\mathcal{A}}^{1}(log\mathcal{I})$, $a,b \in \mathcal{A}=\mathbb{F}[x,y]$:
  \begin{equation}
              \tilde{H}(\omega)=   y^{n-1}(a \delta^{2}- b \delta^{1}) \in Der_{\mathcal{A} }(Log \mathcal{I}), n>1.
       \end{equation}
 \end{proposition}
 \begin{proof}
According to \cite{IV}, any polynomial $h\in \mathcal{A}$ induce the Hamiltonian map $H: \Omega^{1}_{\mathcal{A}}\longrightarrow Der_{\mathcal{A}}$ such that $H(dx) = \{x,.\}$. 
And for all 1-form  $\omega \in \Omega_{\mathcal{A}}^{1}(log\mathcal{I})$ we have $H(\omega) \in Der_{\mathcal{A}}(log \mathcal{I})$. On the other words $H(\Omega_{\mathcal{A}}^{1}(log\mathcal{I}) ) \subset Der_{\mathcal{A}}(log \mathcal{I})$.
          Then this structure induce also  another homomorphism of $\mathcal{A}-$module
         $\tilde{H}\in Hom( \Omega_{\mathcal{A}}^{1}(log\mathcal{I}), Der_{\mathcal{A}}(log \mathcal{I}))$
         that we call the logarithmic Hamiltonian map $\tilde{H}: \Omega_{\mathcal{A}}^{1}(log\mathcal{I}) \longrightarrow   Der_{\mathcal{A}}(log \mathcal{I})$ such that  $\tilde{H}(\omega_{1})  = \tilde{H}(dx) = H(dx) = y^n \partial_{y} = y^{n-1}\delta^{2}$ 
 and  $\tilde{H}(\omega_{2}) =   \tilde{H}(\dfrac{dy}{y}) = \dfrac{1}{y}H(dy) = - y^{n-1} \delta^{1}$.
 Since $\tilde{H}$ is $\mathcal{A}$-linear, we
  deduce that for all 1-forms $\omega= a \omega_{1} + b\omega_{2} \in \Omega_{\mathcal{A} }^{1}(log \mathcal{I})$ with $a, b \in \mathcal{A}$, we have $\tilde{H}(\omega)= a \tilde{H}(\omega_{1}) + b\tilde{H}(\omega_{2})$. Then, the expression of the   logarithmic  Hamiltonian map 
$   \tilde{H}(\omega)=   y^{n-1}(a \delta^{2}- b \delta^{1}) \in Der_{\mathcal{A} }(log \mathcal{I}) ~for ~all~a,~b\in \mathcal{A}.$
       This completes the proof of the above proposition.
 \end{proof}
 By convention we shall consider the following canonical isomorphisms given by:
  \begin{enumerate} 
  \item[$\bullet$] $\mathfrak{L}_{alt}^{0}(\Omega^{1}_{\mathcal{A}}(log\mathcal{I}, \mathcal{A}) \simeq \mathcal{A}$, 
\item[$\bullet$] $\mathfrak{L}_{alt}^{1}(\Omega^{1}_{\mathcal{A}}(log\mathcal{I}, \mathcal{A}) \simeq  Der_{\mathcal{A} }(log \mathcal{I}) $,
 \item[$\bullet$] $\mathfrak{L}_{alt}^{2}(\Omega^{1}_{\mathcal{A}}(log\mathcal{I} , \mathcal{A}) \simeq \wedge^{2} Der_{\mathcal{A} }(log \mathcal{I}) $,
    \item[$\bullet$] $\mathfrak{L}_{alt}^{p}(\Omega^{1}_{\mathcal{A}}(log\mathcal{I} , \mathcal{A}) \simeq 0$, for $p>2$ or $p<0$.
 \end{enumerate} 
Denote by $\textbf{C}^{*}= \mathfrak{L}_{alt}^{*}(\Omega^{1}_{\mathcal{A}}(log\mathcal{I}, \mathcal{A}) = \wedge^{*}Der_{\mathcal{A}}(log \mathcal{I})$. According to  \textit{definition} \ref{5} the cochain complex associated to $\tilde{H}$ is given by
\begin{equation} 
(\textbf{C}^{*},d_{\tilde{H}}^{*}): 0\stackrel{d^{0}_{\tilde{H}}}{\longrightarrow}\mathcal{A}\stackrel{d^{1}_{\tilde{H}}}{\longrightarrow}Der_{\mathcal{A}}(log \mathcal{I})
     \stackrel{d^{2}_{\tilde{H}}}{\longrightarrow} \wedge^{2}Der_{\mathcal{A}}(log \mathcal{I}) \stackrel{d^{3}_{\tilde{H}}}{\longrightarrow} 0. \label{(5)}
\end{equation}     
      The logarithmic Poisson differentials are  $d^{0}_{\tilde{H}} = d^{3}_{\tilde{H}} = 0$ where $d^{1}_{\tilde{H}}$ and $d^{2}_{\tilde{H}}$ are given in the following $lemma$:
 \begin{lemma} \label{41}
 The logarithmic  Poisson differentials  $d^{1}_{\tilde{H}}$ and $d^{2}_{\tilde{H}}$ are given as follows:
 \begin{enumerate}
 \item[1.]  $d^{1}_{\tilde{H}}(a) = y^{n-1}(\delta^{2}a \delta^{1} - \delta^{1}a \delta^{2})\in Der_{\mathcal{A}}(log \mathcal{I})$ for all $a \in \mathcal{A}$,
 \item[2.]  
 $d^{2}_{\tilde{H}}(\overrightarrow{a}) = \left( y^{n-1}(\delta^{1}a^{1} + \delta^{2}a^{2} )-(n-1)y^{n-1}a^{2} \right) \delta^{1}\wedge \delta^{2} \in \wedge^{2}Der_{\mathcal{A}}(log \mathcal{I})$ for all element $\overrightarrow{a}= a^{1}\partial_{x} + a^{2} y\partial_{y}= a^{1}\delta^{1} + a^{2}\delta^{2} \in  Der_{\mathcal{A}}(log \mathcal{I})$.
 \end{enumerate}
 \end{lemma}
\begin{proof}
 \begin{enumerate}
 \item[1.]   
 Let $f\in \mathcal{A}$.  We remember that $Der_{\mathcal{A}}(log \mathcal{I}) = \langle \delta^{1} = \partial_{x}, \delta^{2} = y\partial_{y} \rangle$ and also $\Omega^{1}_{\mathcal{A}}(log \mathcal{I}) = \langle\omega_{1}=dx, \omega_{2}=\dfrac{dy}{y} \rangle$.
For an exact 1-form $dx$, we set $\tilde{H}(dx) = H(dx),$ since $dx$ is also an ordinary 1-form.
 According to the cochain complex $(\ref{(5)})$ we have the relation $d_{\tilde{H}}^{1}(f) = f^{1} \delta^{1} + f^{2}\delta^{2} \in Der_{\mathcal{A}}(log \mathcal{I})$ where $f^{1},f^{1}\in \mathcal{A}$.  We use 
the definition of the logarithmic Poisson differential $d^{i}_{\tilde{H}}$ given in \textit{definition} $(\ref{5})$ which we combine with the $\mathcal{A}-$module homomorphism given in $proposition$ $(\ref{P1})$, the Koszul braket $[.,.]$ of $theorem$ $(\ref{T2})$ and the simplified expression of $[.,.]$ given in $corollary$ $ (\ref{Cor1}) $ to determine respectively $f^{1}$ and $f^{2}$ as follows:
   \begin{align*}
   d_{\tilde{H}}^{1}(f)\lrcorner dx &\ = f^{1}\\
      &\ = \tilde{H}(dx)f\\
   &\ = H(dx)f\\
   &\ = \{x,y\}\partial_{y}f\\
  &\ = y^{n-1}(y\partial_{y}f)
   \end{align*}
 We also have
   \begin{align*}
   d_{\tilde{H}}^{1}(f)\lrcorner\dfrac{dy}{y} &\ = f^{2}\\
   &\ = \tilde{H}(\dfrac{dy}{y})f\\
   &\ =\dfrac{1}{y}H(dy)f\\
  &\ = -y^{n-1}\partial_{x}f
   \end{align*}
 Then
 $d_{\tilde{H}}^{1}(f) = y^{n-1}(y\partial_{y}f \delta^{1} -\partial_{x}f \delta^{2})= y^{n-1}(\delta^{2}f \delta^{1} -\delta^{1}f \delta^{2})\in Der_{\mathcal{A} }(Log \mathcal{I})$.
 \item[2.] Let $f \in Der_{\mathcal{A} }(Log \mathcal{I}) = \langle \delta^{1} = \partial_{x} ; \delta^{2} = y\partial_{y} \rangle_{\mathcal{A} }$, then $ f = f_{1}\partial_{x} + f_{2}y\partial_{y}$. So we have: 
 \begin{align*}
 d_{\tilde{H}}^{2}(f) &\ = d_{\tilde{H}}^{2}(f_{1}\partial_{x} + f_{2}y\partial_{y}).\\
 &\ = d_{\tilde{H}}^{2}(f_{1}\partial_{x}) + d_{\tilde{H}}^{2}(f_{2}y\partial_{y})\in \wedge^{2}Der_{\mathcal{A} }(Log \mathcal{I}).  \\
 \end{align*} There exist $a,b\in \mathcal{A}$ such that $ d_{\tilde{H}}^{2}(f_{1}\partial_{x})=a \partial_{x}\wedge y \partial_{y}  $ and $ d_{\tilde{H}}^{2}(f_{2}y\partial_{y}) =  b \partial_{x}\wedge y\partial_{y}$.
We remember that $\delta^{i}\lrcorner \omega_{j}=\delta_{ij}$. We remember to the definition of the logarithmic Poisson differential $d^{i}_{\tilde{H}}$ given in \textit{definition} $(\ref{5})$ which we combine with the $\mathcal{A}-$module homomorphism given in $proposition$ $(\ref{P1})$, the Koszul braket $[.,.]$ of $theorem$ $(\ref{T2})$ and the simplified expression of $[.,.]$ given in $corollary$ $ (\ref{Cor1}) $
 to compute  $a$  as follows:
 \begin{align*}
 a &\ = d_{\tilde{H}}^{2}(f_{1}\partial_{x}) ( dx ,\dfrac{dy}{y} ).\\
  &\ =-\tilde{H}(dx)f_{1}\partial_{x}\lrcorner\dfrac{dy}{y}+\tilde{H}(\dfrac{dy}{y})(f_{1}\partial_{x}\lrcorner dx) -f_{1} \partial_{x}\lrcorner [dx, \dfrac{dy}{y}]\\
 &\ =-\tilde{H}(dx)f_{1}\partial_{x}\lrcorner\dfrac{dy}{y}+\tilde{H}(\dfrac{dy}{y})(f_{1}\partial_{x}\lrcorner dx) -f_{1} \partial_{x}\lrcorner d(\dfrac{1}{y}\{x,y\})\\
 &\ = y^{n-1}\partial_{x}f_{1}
 \end{align*}
 With the similar method we also compute $b$ as follows:
 \begin{align*}
 b &\ = d_{\tilde{H}}^{2}(f_{2}y\partial_{y}) ( dx ,\dfrac{dy}{y} )\\
  &\ =-\tilde{H}(dx)f_{2}y\partial_{y}\lrcorner\dfrac{dy}{y}+\tilde{H}(\dfrac{dy}{y})(f_{2}y\partial_{y}\lrcorner dx) -f_{2} y\partial_{y}\lrcorner [dx, \dfrac{dy}{y}]\\
 &\ =-\tilde{H}(dx)f_{2}y\partial_{y}\lrcorner\dfrac{dy}{y}+\tilde{H}(\dfrac{dy}{y})(f_{2}y\partial_{y}\lrcorner dx) -f_{2} y\partial_{y}\lrcorner d(\dfrac{1}{y}\{x,y\})\\
 &\ = y^{n-1}\partial_{y}f_{2}-(n-1)y^{n-1}f_{2}
 \end{align*}
 \end{enumerate}
 Then 
  $ d_{\tilde{H}}^{2}(\overrightarrow{a}) = \left( y^{n-1} (\partial_{x}a^{1} + y\partial_{y}a^{2}) -(n-1)y^{n-1}a^{2}\right)\delta^{1}\wedge \delta^{2}\in \wedge^{2}Der_{\mathcal{A} }(Log \mathcal{I})$, equivalent to   $ d_{\tilde{H}}^{2}(\overrightarrow{a}) = \left( y^{n-1} (\delta^{1}a^{1} + \delta^{2}a^{2}) -(n-1)y^{n-1}a^{2}\right)\delta^{1}\wedge \delta^{2}\in \wedge^{2}Der_{\mathcal{A} }(Log \mathcal{I})$.
 This completes the proof of the above $lemma$.
\end{proof}
We deduce a relationship between the logarithmic Hamiltonian operator  and the Lichnerowicz differential initialy define by $d_{\pi} = [\pi,.]$ associated to bivector $\pi$.
\begin{corollary}
Let $\pi = y^n \partial_{x} \wedge \partial_{y}$ with $n>1$; $[.,.]_{SN}$  the Schouten - Nijenhuis bracket and the logarithmic differential $\tilde{d}: \mathcal{A} \longrightarrow \Omega_{\mathcal{A}}(log \mathcal{I})$, $f\mapsto \tilde{d}f = \partial_{x}fdx + y\partial_{y}f\dfrac{dy}{y}$. For all $f\in \mathcal{A}$ we have $[\pi , f]_{SN} = -\tilde{H}(\tilde{d}f)$.
\end{corollary}

\begin{proof}
We remember that, the standard Lichnerowicz differential associated to the logarithmic Hamiltonian operator $\tilde{H}$ is given by $ d_{\tilde{H}}^{1}(f) = [\pi , f]_{SN}$(see \cite{AL}). For all $f\in \mathcal{A}$,
 \begin{align*}
 \tilde{H}(\tilde{d}f) &\ = \tilde{H}(\partial_{x}f dx + y\partial_{y}f\dfrac{dy}{y} )\\
  &\ = \partial_{x}f \tilde{H}( dx) + y\partial_{y}f \tilde{H}(\dfrac{dy}{y} )\\
 &\ = \partial_{x}f(y^n \partial_{y}) + y\partial_{y}f (-y^{n-1}\partial_{x} )\\
 &\ = y^{n-1}(\partial_{x}f \delta^{2} - y\partial_{y}f \delta^{1} )\\
 &\ = y^{n-1}(\delta^{1}f \delta^{2} - \delta^{2}f \delta^{1} )\\
  &\ = -d_{\tilde{H}}^{1}(f)
 \end{align*}
 Then  $ d_{\tilde{H}}^{1}(f) = [\pi , f]_{SN} = -\tilde{H}(\tilde{d}f)$. This completes the proof.
\end{proof}
We just  prove that the following diagram is commutative (i.e $-\tilde{H}\circ \tilde{d} = d_{\tilde{H}}^{1}$).
\begin{eqnarray}
 \xymatrix{\mathcal{A}\ar[r]^{\tilde{d}} \ar[dr]_{d_{\tilde{H}}^{1}}&\Omega_{\mathcal{A}}(log \mathcal{I})\ar[d]^{-\tilde{H}}\\
                                   &             Der_{\mathcal{A}}(log \mathcal{I})}
 \end{eqnarray}
Let $\partial^{0}$, $\partial^{1}$, $\partial^{2}$ be the canonical isomorphisms  respectively given by the following:
    $\partial^{0}:= Id_{\mathcal{A}}:    \mathcal{A} \longrightarrow  \mathcal{A}$, $a \longmapsto \partial^{0} (a) = a$(an identity map);\\
     $ \partial^{1}:   \mathcal{A}^{2}    \longrightarrow Der_{\mathcal{A}}(log \mathcal{I}), ~~
                          ( f^1 , f^2) \longmapsto \partial^{1}(f^1, f^2) = f^1 \delta^{1}  + f^2 \delta^{2} $;\\
     $\partial^{2}:  \mathcal{A}    \longrightarrow  \wedge^{2}Der_{\mathcal{A}}(log \mathcal{I}),~~
                                                                                   a \longmapsto \partial^{2}(a) = a\delta^{1}  \wedge \delta^{2}.$
    So its follows that  $\partial^{1}\circ d^{1} = d^{1}_{\tilde{H}} \circ \partial^{0}$, $\partial^{2}\circ d^{2} = d^{2}_{\tilde{H}} \circ \partial^{1}$ 
    and we consider the following differentials:
     \begin{enumerate}
      \item[i)]  $d^{1}(a) = y^{n-1}(y\partial_{y}a, -\partial_{x}a)$ for all $a \in \mathcal{A}$,
      \item[ii)]  
      $d^{2}(\overrightarrow{a}) = y^{n-1}(\partial_{x}a^{1} + y\partial_{y}a^{2} )-(n-1)y^{n-1}a^{2}$ for all $\overrightarrow{a}= (a^{1}, a^{2})\in \mathcal{A}\times \mathcal{A}$.
      \end{enumerate}
     this makes commutative the following  diagram: 
 
  \[ \xymatrix{ 0\ar[r]^{d^0}& \mathcal{A} \ar[d]^{\partial^{0}}\ar[r] ^{d^1} &\mathcal{A}\times \mathcal{A}\ar[d]^{\partial^1}\ar[r]^{d^2}&\mathcal{A} \ar[d]^{\partial^2}\ar[r] ^{d^3}&0\\
      0\ar[r] & \mathcal{A} \ar[r]&Der_{\mathcal{A}}(log \mathcal{I})\ar[r]&\bigwedge^{2} Der_{\mathcal{A}}(log \mathcal{I})\ar[r]&0}
     \] 
 The following lemma informs us about the corresponding cochain complex.  
 \begin{lemma} $\label{43}$
  $ \xymatrix{ 0 \ar[r]^{d^{0}} & \mathcal{A} \ar[r]^{d^{1}}  & \mathcal{A}\times \mathcal{A}
            \ar[r]^{d^{2}}  & \mathcal{A} \ar[r]^{d^{3}}  & 0 }
            $
           is the logarithmic  cochain complex.
 \end{lemma}
     \begin{proof}
For all $a\in \mathcal{A}$, we have $ d^{i+1}\circ  d^{i}(a)=0.$
\end{proof}
According to the above cochain complex, we define the logarithmic Poisson cohomology as follows:
 \begin{definition}
The
             $(i-1)^{th}$  logarithmic Poisson cohomology group of  the logarithmic Poisson algebra $\tilde{\mathcal{P}}=(\mathcal{A};\{.,.\}, \tilde{H})$ along $\mathcal{I}$ is $H_{log}^{i-1}(\tilde{\mathcal{P}}) = \dfrac{Z^{i}(\textbf{C}^{*},d^{*})}{B^{i}(\textbf{C}^{*},d^{*})}$ 
        where
          $Z^{i}(\textbf{C}^{*},d^{*}) = Kerd^{i}$ is the space of Poisson cocycles of order $i$ and  $B^{i}(\textbf{C}^{*},d^{*}) = Imd^{i-1}$ is the space of
                   Poisson coboundary of order  $i$ for all $i\geqslant 1$.
            \end{definition}  
In the following, denote by $\tilde{\mathcal{P}} = (\mathcal{A}=\mathbb{F}[x,y];\{x,y\} = y^n; \tilde{H})$ the logarithmic Poisson algebra associated to a degenerate Poisson bivector $\pi=y^n\partial _{x} \wedge \partial _{y}$ and $\tilde{H}$ denote the logarithmic Hamiltonian map.
\section{Logarithmic Poisson cohomology of $\tilde{\mathcal{P}}$ along $\mathcal{I}$}
             \subsection{Computation of $H_{log}^{0}(\tilde{\mathcal{P}})$}
             \begin{proposition} The  $0^{th}$  logarithmic Poisson cohomology group of $\tilde{\mathcal{P}}$ along $\mathcal{I}$ is
             $H_{log}^{0}(\tilde{\mathcal{P}}) \simeq \mathbb{F}$.
             \end{proposition}            
\begin{proof}
Let $a\in \mathcal{A}\cap Z^{1}(\textbf{C}^{*},d^{*})$. Thus $d^{1}(a)=0$ if and only if $\partial_{x}a=\partial_{y}a=0$. It follows that $a\in \mathbb{F}$.  Therefore $Z^{1}(\textbf{C}^{*},d^{*})=\mathbb{F}$. Since $Z^{1}(\textbf{C}^{*},d^{*})= B^{1}(\textbf{C}^{*},d^{*})\oplus H_{log}^{0}(\tilde{\mathcal{P}})$ and $B^{1}(\textbf{C}^{*},d^{*})=0$, we deduce that $H_{log}^{0}(\tilde{\mathcal{P}})=\mathbb{F}$. It is the center of the logarithmic Poisson algebra $\tilde{\mathcal{P}}$.
\end{proof}
             \subsection{Computation of $H_{log}^{1}(\tilde{\mathcal{P}})$}
             \begin{lemma} $\label{(3)}$
             Let $\mu: E\longrightarrow F $ a monomorphism of vector spaces. Then for any $p$ linear  subspaces $E_{i}$ of $E$,
          $\mu(E_{1}\oplus...\oplus E_{p})= \mu (E_{1})\oplus...\oplus\mu ( E_{p})$.
             \end{lemma}
          In the following, we determine the Poisson cocycle of order 2 noted $Z^{2}(\textbf{C}^{*},d^{*}) $ given  by.
             \begin{lemma}
              $Z^{2}(\textbf{C}^{*},d^{*}) = \mu (\overset{n-2}{\underset{i=0}{\bigoplus}}y^{j} \mathbb{F}[x] )\oplus \mu(y^{n-1} \mathbb{F}[x,y] )\oplus(\mathbb{F}_{n-1}[y]\times 0 )\oplus(y^{n} \mathbb{F}[y]\times 0 )$.
             \end{lemma}            
            \begin{proof}
            Let $(a,b)\in \mathcal{A}$. We have $(a,b)\in Z^{2}(\textbf{C}^{*},d^{*})$ if and only if $y^{n-1}\left(\partial_{x}a + y\partial_{y}b -(n-1)b\right)=0$. Therefore $\partial_{x}a = (n-1)b - y\partial_{y}b$. It follows that $ a=\int \left((n-1)b - y\partial_{y}b\right)dx + \alpha(y)$. We consider the following   $\mathbb{F}$-linear map
            \begin{equation}
            \mu : \mathcal{A} \longrightarrow \mathcal{A}\times \mathcal{A}, b\mapsto \left(\int ((n-1)b - y\partial_{y}b)dx; b\right).
            \end{equation}
            This implies that $ Z^{2}(\textbf{C}^{*},d^{*})= \{ \left(\int ((n-1)b - y\partial_{y}b)dx; b\right) + (\alpha(y),0); b\in \mathcal{A}, \alpha(y)\in \mathbb{F}[y] \}$ for all $n>1$.
           We remember that $\mathcal{A} =  \overset{n-2}{\underset{i=0}{\bigoplus}}y^{j} \mathbb{F}[x] \oplus y^{n-1} \mathbb{F}[x,y]$.
          According to $lemma$ $\ref{(3)}$, we deduce the explicit expression of  $Z^{2}(\textbf{C}^{*},d^{*})$ by:
            \begin{align*}
             Z^{2}(\textbf{C}^{*},d^{*}) &\ = \mu (\mathcal{A} )\oplus ( \mathbb{F}[y]\times 0)\\
             &\ = \mu \left(\overset{n-2}{\underset{i=0}{\bigoplus}}y^{j} \mathbb{F}[x] \oplus y^{n-1} \mathbb{F}[x,y] \right)\oplus ( \mathbb{F}[y]\times 0)\\
             &\ = \mu \left(\overset{n-2}{\underset{i=0}{\bigoplus}}y^{j} \mathbb{F}[x] \right)\oplus \mu\left(y^{n-1} \mathbb{F}[x,y] \right)\oplus(\mathbb{F}_{n-1}[y]\times 0 )\oplus(y^{n} \mathbb{F}[y]\times 0 ).\\
             \end{align*}
    This completes the proof.
            \end{proof}
In the following, we investigate the relation  that exists between $Z^{2}(\textbf{C}^{*},d^{*})$ and $B^{2}(\textbf{C}^{*},d^{*})$. 
\begin{lemma}  $\label{L5}$
For all $n>1$,
              $B^{2}(\textbf{C}^{*},d^{*})\cap (\overset{n-2}{\underset{i=0}{\bigoplus}}y^{j} \mathbb{F}[x]) = 0$ and also $B^{2}(\textbf{C}^{*},d^{*})\cap (\mathbb{F}_{n-1}[y]\times 0) = 0$.
             \end{lemma}
\begin{proof}
Remember that $d^{2}\circ d^{1}=0$ implies that $B^{2}(\textbf{C}^{*},d^{*})\subset Z^{2}(\textbf{C}^{*},d^{*})$.
Furthermore, for all $a\in \mathcal{A}$ $d^{1}(a) = (\psi^{1},\psi^{2})$ where 
 the degree of $\psi^{1}$ in $y$ is greater  than $(n-1)$. Therefore   $B^{2}(\textbf{C}^{*},d^{*})\cap (\overset{n-2}{\underset{i=0}{\bigoplus}}y^{j} \mathbb{F}[x]) = 0$ and also $B^{2}(\textbf{C}^{*},d^{*})\cap (\mathbb{F}_{n-1}[y]\times 0) = 0$.
\end{proof}    
            \begin{lemma}$\label{L6}$         
\begin{enumerate}
\item[1.] $B^{2}(\textbf{C}^{*},d^{*}) \cap( y^{n} \mathbb{F}[y]\times 0 )= y^{n} \mathbb{F}[y]\times 0$,
\item[2.] $B^{2}(\textbf{C}^{*},d^{*})\cap \mu\left(y^{n-1} \mathbb{F}[x,y]\right)  =  \mu\left(y^{n-1} \mathbb{F}[x,y]\right)$.
\end{enumerate}
            \end{lemma}
   \begin{proof} $\label{PP}$\\
\begin{enumerate}
\item[1.]     We will  show that $ (y^{n} \mathbb{F}[y]\times 0 ) \subset B^{2}(\textbf{C}^{*},d^{*})$. Let us $\overrightarrow{b}=(y^n b_{1}(y), 0) \in y^{n} \mathbb{F}[y]\times 0$. We are looking for an element $a\in \mathcal{A}$ such that $d^1(a)= \overrightarrow{b}$. We have $d^1(a)= \overrightarrow{b}$ equivalent to $\partial_{y}a = \partial_{y}b_{1}(y) $ and $\partial_{x}a = 0$ which imply  that $a = \int b_{1}(y)dy + \alpha(x)$ and $a\in  \mathbb{F}[y]$ respectively. We get  $a = \int b_{1}(y)dy + \alpha_{0}$ and we verify that we have $d^1\left(\int b_{1}(y)dy + \alpha_{0}\right)= \overrightarrow{b}$, for all $\alpha_{0}\in \mathbb{F}$. Hence we just prove that $\overrightarrow{b}\in B^{2}(\textbf{C}^{*},d^{*})$. Therefore $y^{n} \mathbb{F}[y]\times 0  \subset B^{2}(\textbf{C}^{*},d^{*})$. Then $    B^{2}(\textbf{C}^{*},d^{*}) \cap( y^{n} \mathbb{F}[y]\times 0 )= y^{n} \mathbb{F}[y]\times 0 $. 
 \item[2.] We need to prove that $\mu(y^{n-1} \mathbb{F}[x,y] )\subset  B^{2}(\textbf{C}^{*},d^{*})$. So we take $ g = \mu(y^{n-1}b)$ an element of $ \mu\left(y^{n-1}\mathbb{F}[x,y]\right) $. For all $f\in \mathcal{A}$, $d^1(f)=g $ if and only if we have equation $y^{n-1}(y\partial_{y}f, -\partial_{x}f) = \left( \int y^{n}\partial_{y}bdx, y^{n-1}b\right)$, which means that $f=-\int bdx + \beta_{0}$. Furthermore, there exists $f=-\int bdx + \beta_{0}\in \mathcal{A}$ with $\beta_{0}\in \mathbb{F}$ such that $d^1(f)=g$. Therefore we deduce that $g\in B^{2}(\textbf{C}^{*},d^{*})$. Thus $\mu(y^{n-1} \mathbb{F}[x,y] )\subset  B^{2}(\textbf{C}^{*},d^{*})$, this completes exactly the proof of $B^{2}(\textbf{C}^{*},d^{*}) \cap \mu(y^{n-1} \mathbb{F}[x,y] )  =  \mu(y^{n-1} \mathbb{F}[x,y] )$.
\end{enumerate}
This completes the proof.
   \end{proof}
  
  \begin{proposition} 
  The  $1^{th}$  logarithmic Poisson cohomology group of $\tilde{\mathcal{P}}$ along $\mathcal{I}$ is given by
  $H_{log}^{1}(\tilde{\mathcal{P}}) \simeq  \mu (\overset{n-2}{\underset{i=0}{\bigoplus}}y^{i} \mathbb{F}[x])\oplus ( \mathbb{F}_{n-1}[y]\times 0)$ for all $n>1$.
    \end{proposition} 
\begin{proof}
It come from  the above $lemma$ $\ref{L5}$ and $lemma$ $\ref{L5}$ that we have the following:\\
 $Z^{2}(\textbf{C}^{*},d^{*})  = \mu (\overset{n-2}{\underset{i=0}{\bigoplus}}y^{j} \mathbb{F}[x] )\oplus  (\mathbb{F}_{n-1}[y]\times 0 )\oplus B^{2}(\textbf{C}^{*},d^{*})\simeq H_{log}^{1}(\tilde{\mathcal{P}}) \oplus B^{2}(\textbf{C}^{*},d^{*}).$ 
 This completes the proof.
\end{proof} 
  
 \subsection{Computation of $H_{log}^{2}(\tilde{\mathcal{P}})$}
 \begin{lemma} For all $n>1$,
$\mathcal{A}= B^{3}(\textbf{C}^{*},d^{*})\oplus \overset{n-2}{\underset{i=0}{\bigoplus}}y^{i} \mathbb{F}[x]$.
 \end{lemma}
\begin{proof}
Let $f\in \mathcal{A}=\mathbb{F}[x,y]$. Decompose $f$ by $f_{0}(x)+ yf_{1}(x) +...+y^{n-2}f_{_{n-2}}(x)+y^{n-1}f_{_{n-1}}(x,y)  $  and $\mathcal{A}$ by $\mathcal{A} = \mathbb{F}[x]\oplus y\mathbb{F}[x]\oplus ...\oplus y^{n-2} \mathbb{F}[x] \oplus y^{n-1}\mathbb{F}[x,y]$.
 We have $f\in B^{3}(\textbf{C}^{*},d^{*})$ if and only if there exists $(a,b)\in \mathcal{A}^{2}$ such that $d^{2}(a,b) = f$. On the other words, $f\in B^{3}(\textbf{C}^{*},d^{*})$ if and only if there exists $(a,b)\in \mathcal{A}^{2}$ such that the following system holds: 
   \begin{eqnarray} \label{E8}
  \left\{
  \begin{array}{rl}
   \partial_{x} a + y\partial_{y}b - (n-1)b=f_{_{n-1}}(x,y)    \\
   f_{i}(x,y) = 0, ~for~ all ~i\in \{0;1;...;n-2;n\}           
  \end{array}
   \right.
  \end{eqnarray}
  It  follows from the above system that we have $a= \int \left(f_{_{n-1}}(x,y) + (n-1)b-y\partial_{y}b\right)dx + \beta(y)$. Particularly, if we take $b=f_{_{n-1}}(x,y)$ and $\beta(y)=0$, we have $a=\int \left(nb-y\partial_{y} b\right)dx$. It is clear that if the polynomial given by  $f_{0}(x)+ yf_{1}(x) +...+y^{n-2}f_{_{n-2}}(x)$ is different to zero,  it can not to be in $B^{3}(\textbf{C}^{*},d^{*})$. Then we deduce that $\mathcal{A}= d^{2}(\mathcal{A} ^{2})\oplus \overset{n-2}{\underset{i=0}{\bigoplus}}y^{i} \mathbb{F}[x]$. On the other words $\mathcal{A}= B^{3}(\textbf{C}^{*},d^{*})\oplus \overset{n-2}{\underset{i=0}{\bigoplus}}y^{i} \mathbb{F}[x]$.
\end{proof}  
  \begin{proposition}
  The  $2^{th}$  logarithmic Poisson cohomology group of $\tilde{\mathcal{P}}$ along $\mathcal{I}$ is given by
  $H_{log}^{2}(\tilde{\mathcal{P}}) \simeq \overset{n-2}{\underset{i=0}{\bigoplus}}y^{i} \mathbb{F}[x]$.
   \end{proposition}
   \begin{proof}
   Since $Z^{3}(\textbf{C}^{*},d^{*})=\mathcal{A}$, it follows that  $Z^{3}(\textbf{C}^{*},d^{*})= B^{3}(\textbf{C}^{*},d^{*})\oplus \overset{n-2}{\underset{i=0}{\bigoplus}}y^{i} \mathbb{F}[x]$. Then we get $H_{log}^{2}(\tilde{\mathcal{P}}) \simeq \overset{n-2}{\underset{i=0}{\bigoplus}}y^{i} \mathbb{F}[x]$. 
This completes the proof.
   \end{proof}
   
  The above result obtained for the $2^{th}$ logarithmic Poisson cohomology group is very interesting.
  It parametrizes the obstruction space for deformation quantization.
   Because $H_{log}^{2}(\tilde{\mathcal{P}})\neq 0$, the Poisson bivector $\pi = y^n\partial_{x}\wedge \partial_{y}$, $n>1$ is not rigid; equivalently, this lack of rigidity means that this Poisson bivector admits non-trivial deformations. \\
      
     \begin{proposition}
     The  $k^{th}$  logarithmic Poisson cohomology group of $\tilde{\mathcal{P}}$ along $\mathcal{I}$ is given by
     $H_{log}^{k}(\tilde{\mathcal{P}}) \simeq 0$ for all $k>2$.
      \end{proposition}
      \begin{proof}
          For all $k>3$ we have $d^k=0$, thus 
         $B^{k}(\textbf{C}^{*},d^{*})= Z^{k}(\textbf{C}^{*},d^{*})=0$. Then we have clearly
           $H_{log}^{k}(\tilde{\mathcal{P}}) \simeq 0$ for all  $k>2$.
      \end{proof}
   The propositions in this section complete the proof of the following $theorem$.
 \begin{theorem} $\label{T3}$\\
 Let $\mathcal{A}=\mathbb{F}[x,y]$ and  $\tilde{\mathcal{P}}=(\mathcal{A};\{x,y\}=y^n, \tilde{H})$ the logarithmic Poisson algebra. The logarithmic Poisson cohomology of a degenerate Poisson bivector $\pi = y^n\partial_{x}\wedge \partial_{y}$ is given by the following:
  \begin{enumerate}
    \item[1.]  $H_{log}^{0}(\tilde{\mathcal{P}}) \simeq \mathbb{F}$,
  \item[2.]  $H_{log}^{1}(\tilde{\mathcal{P}}) \simeq  \mu (\overset{n-2}{\underset{i=0}{\bigoplus}}y^{i} \mathbb{F}[x])\oplus ( \mathbb{F}_{n-1}[y]\times 0)$, where $\mu$ is the  $\mathbb{F}$-linear map  $\mu : \mathcal{A} \longrightarrow \mathcal{A}\times \mathcal{A}$, such that $b \mapsto \left(\int ((n-1)b - y\partial_{y}b)dx; b\right)$,  \item[3.] $H_{log}^{2}(\tilde{\mathcal{P}}) \simeq \overset{n-2}{\underset{i=0}{\bigoplus}}y^{i} \mathbb{F}[x]$,
      \item[4.]  $H_{log}^{k}(\tilde{\mathcal{P}}) \simeq 0$, for all $k>2$. 
  \end{enumerate}
 \end{theorem}
 According to \cite{BI} and this work, we have the following corollary.
\begin{corollary} Let us $\mathcal{A}=\mathbb{F}[x,y]$, $\mathcal{P}=(\mathcal{A};\{x,y\}=y^n, H)$ the classical Poisson algebra,  $\tilde{\mathcal{P}}=(\mathcal{A};\{x,y\}=y^n, \tilde{H})$ the corresponding logarithmic Poisson algebra associated to a degenerate Poisson bivector $\pi = y^n\partial_{x}\wedge \partial_{y}$ and $H_{Pois}^{k}(\mathcal{P})$ the classical Poisson cohomology.
Then we have  $H_{log}^{k}(\tilde{\mathcal{P}}) \simeq H_{Pois}^{k}(\mathcal{P})$ for all $k\geqslant 0$.
 \end{corollary}
 \section{Conclusion}
We have met the  challenge of investigating some properties of the logarithmic cohomological invariants induced by
the Poisson bivector $\pi = y^n\partial_{x}\wedge \partial_{y}$. Previously underexposed, and of performing
explicit computation of logarithmic Poisson cohomology. It comes from this that, the logarithmic Poisson cohomology group $H_{log}^{*}(\tilde{\mathcal{P}})$ is
 agree with it classical Poisson cohomology group $H_{Pois}^{*}(\mathcal{P})$.

\section*{Declarations}
 We declare that the attached manuscript is original, has not been published previously, and is not currently under review by another journal.\\
 
\textbf{ Funding statement:} This research did not receive any specific grant from funding agencies in the public commercial, or not-for-profit sectors.\\

 \textbf{Conflict interest:} The authors have no relevant financial or non-financial interests to disclose.\\
 
\textbf{Author contributions: } The authors contributed to the study's conception and design, performed the research, and wrote the manuscript. All authors read and approved the final manuscript.\\

\textbf{Data availability statement:} The data that support the finding of this study are available in the references of this article via Digital Object Identifiers(DOI). Some references, however such as thesis work and publication from  older journal editions, do not have DOIs and are available from the corresponding author upon reasonable request.\\

\bibliography{sn-bibliography}

\begin{thebibliography}{25}
%
	\bibitem{DJ1} J. Dongho. {\it Logarithmic Poisson Structures: Cohomological Invariants and Pre-Quantification}, University of Angers, 2012
	%
	\bibitem{FB} B. Frederic. {\it Poisson Structures on the polynomial algebras, cohomology and deformations},Claude Roger University - Lyon 1, Lyon, 13-th november 2009
	%
	\bibitem{BI} B. Iskamle, J. Dongho, B. Ndombol. {\it On some properties of Poisson cohomology: Example of calculation on a Poisson structure}, Proceedings  of the American Mathematical Society,(2025) doi	= 10.1090/proc/17177.
	%
	\bibitem{AJ} A. Jeanneret,  D. Lines. {\it Invitation to Algebraic Topology Volume I: Homology}, \'{E}ditions C\'{e}padu\`{e}s~ 1 (2014) 
	%
	\bibitem{CL} C. Laurent-Gengoux, A. Pichereau, P. Vanhaecke. {\it Poisson structure}, Grundlerhen der Mathematischen Wissenschaften, Springer~37 (2013), doi	= 10.1007/9783-642-31090-4
	%
	\bibitem{AL} A. Lichnerowicz, {\it Manifold of Poisson and their associated Lie algebras},Journal of differential geometry, Lehigh University.~12 (1977), no.2, 253-300
	%
	\bibitem{PM} P. Monnier. {\it Poisson cohomology in dimension two}, Israel journal of mathematics, Springer, vol 129,189-207, no. 1, 2002, doi=10.1007/BF02773163
	%
	\bibitem{NN} N. Nakanishi. {\it Poisson cohomology of plane quadratic Poisson structures}. Publications of the research institute for mathematical sciences,  vol 33, n° 1, 73-89  (1997), 	       doi= 10.2977/PRIMS/1195145534
	%
	\bibitem{PA} A. Pichereau. {\it Poisson (Co)homology and isolated singularities in small dimensions, with an application in deformation theory}, Poitiers: theses.fr/2006POIT2354 (2006)
	%
	\bibitem{CRPV} C. Roger and P. Vanhaecke. {\it Poisson cohomologyof the affine plane}, in preparation
	%
	\bibitem{KS} K. Saito. {\it Theory of logarithmic differential forms and logarithmic vector fields}, IJournal of the Faculty of Science, Univiversity of Tokyo Sect 1A, Mathematics. 27 (2024), no.~2, 265-291, 	  doi= 10.15083/00039776
	%
	\bibitem{IV} I. Vaisman. {\it On the geometric quantization of poisson manifolds}, Journal of mathematical physics, vol 32(1991), no.12, 3339-3345, 	  doi= 10.1063/1.529447
	%
	\end{thebibliography}

\end{document}